\documentclass[12pt]{amsart}
\usepackage[english]{babel}
\usepackage{amsmath}
\usepackage{amsthm}
\usepackage{amsfonts} 
\usepackage{hyperref}
\usepackage{epsfig}
\usepackage{xcolor}

\numberwithin{equation}{section}

\allowdisplaybreaks

\newtheorem{theorem}{Theorem}
\newtheorem{proposition}{Proposition}[section]

\renewcommand{\epsilon}{\varepsilon}

\newcommand{\1}[1]{{\mathbf 1}{\{#1\}}}
\newcommand{\R}{\mathbb{R}}

\newcommand{\Z}{\mathbb{Z}}
 
\newcommand{\PR}{\mathbb{P}}
\newcommand{\Pb}{\overline{\mathbb{P}}}

\newcommand{\ES}{\mathbb{E}}

\newcommand{\D}{\mathcal{D}}
\newcommand{\F}{\mathcal{F}}

\title{Random Memory Walk}
\date{}
\author[A.~Fribergh]{Alexander Fribergh}
\address{Universit\'e de Montr\'eal, Department of Mathematics and Statistics}
\email{fribergh@dms.umontreal.ca}
\author[D.~Kious]{Daniel Kious}
\address{University of Bath, Department of Mathematical Sciences}
\email{d.kious@bath.ac.uk}
\author[V.~Sidoravicius]{Vladas Sidoravicius}
\author[A.~Stauffer]{Alexandre Stauffer}
\address{Universit\`a Roma Tre, Dipartimento di Matematica e Fisica; University of Bath, Department of Mathematical Sciences, supported by EPSRC Fellowship EP/N004566/1.}
\email{astauffer@mat.uniroma3.it}

\begin{document}

\begin{abstract}
   We present a simple model of a random walk with partial memory, which we call the \emph{random memory walk}. We introduce this model motivated by the belief that it mimics the behavior of
   the once-reinforced random walk in high dimensions and with small reinforcement. We establish the transience of the random memory walk in dimensions three and higher, and show that its scaling limit is a Brownian motion.
\end{abstract}

\maketitle
\begin{center}
\small{\it To the memory of Vladas Sidoravicius}
\end{center}
\section{Introduction}
This short paper started in discussions between the authors during a visit to NYU Shanghai. 
The model we study here, which we call the \emph{random memory walk}, was suggested by Vladas as a way to interpolate between the more well understood case of a random walk with bounded memory (similar to the so-called
\emph{senile random walk} \cite{holmes1,holmes2}) and the challenging model of once-reinforced random walk, which Vladas was fascinated about. 
In this paper we will discuss the behavior of the random memory walk. It turned out that the analysis of this model is quite simple once one looks at it from the right point of view. 

We start this paper by explaining the once-reinforced random walk, some related models,  and the main questions in this area, which motivated us (and, in particular, Vladas) to 
look at this model. Then we explain the link between the once reinforced random walk and the random memory walk, and proceed to the analysis of the latter model.

\subsection{Once-reinforced random walk (ORRW)}
This is one of such models whose definition is very simple but whose analysis is far from trivial. 
In fact, despite being introduced about three decades ago, the behavior of the ORRW on $\mathbb{Z}^d$ is still not well understood, even at an intuitive level, 
and there are essentially no rigorous result about it.


We start defining the ORRW.
Consider an infinite, locally finite graph $G=(V,E)$ with vertex set $V$ and (non-oriented) edge set $E$, and with a distinguished vertex, called the \emph{origin}, that we denote $0$. 
Given a reinforcement parameter $\delta>0$, the ORRW $(X_n)_{n\ge0}$ is defined by the following dynamics. 
Start at time $0$ by placing the random walk at the origin (i.e., $X_0=0$) and by assigning weight $1$ to every edge of $E$. 
Then, at time $n\ge 1$, the random walk jumps to one of its neighbors with a probability proportional to the weight of each edge between them. 
Note that the first jump of the random walk is to a neighbor of $0$ chosen uniformly at random. 
Whenever the walk jumps across an edge $e$ for the first time, the weight of $e$ is updated from $1$ to $1+\delta$, and then the weight of $e$ is never updated again from that time onwards.

More formally, let  $E_n$ be the collection of  edges crossed by the random walk up to time $n$,  that is
\begin{align}\label{defEn}
   E_n:=\left\{e\in E: \exists  k\in \{1,\dots,n\} \text{ s.t.~}\{X_{k-1},X_k\}=e\right\}.
\end{align}
At time $n\in\mathbb{N}$ and on the event $\left\{X_n=x \right\}$ with $x\in V$,  the walk jumps  to a neighbor $y\sim x$ with conditional probability
\[
\PR\left(\left.X_{n+1}=y\right|\mathcal{F}_n\right)=\frac{1+\delta\1{\{x,y\}\in E_n}}{\sum_{z:z\sim x}1+\delta\1{\{x,y\}\in E_n}},
\]
where $\left(\mathcal{F}_n\right)$ is the filtration generated by the history of $(X_n)$, i.e.~$\mathcal{F}_n=\sigma(X_k,0\le k\le n)$ for any integer $n\ge0$.

This random walk thus favors edges that it has already crossed in the past (which, as usual, we call the \emph{range} of the walk), 
and $\delta$ regulates the strength with which the random walk favors its range. 
Intuitively, one could say that, as the random walk grows its range, it interacts with it by experiencing a {\it drift inwards} whenever it tries to move out of its range.
In other words, the random walk is attracted to traverse edges that it has already traversed in the past, creating some sort of a small trap for the walk. 

\subsection{Expected behavior of ORRW on $\Z^d$}
It is particularly interesting to study the ORRW on $\mathbb{Z}^d$, $d\ge 2$, where interesting conjectures have been made. 
The ORRW was introduced by Davis \cite{Davis} in 1990 as a simplification of the linearly edge-reinforced random walk, 
which was defined by Coppersmith and Diaconis in the late eighties. 
Coppersmith and Diaconis conjectured that the linearly edge-reinforced random walk undergoes a phase transition between recurrence and transience, but this was only established
about 25 years later in a sequence of papers \cite{ST,DST,ACK,SZ}. 

When defining the ORRW, Davis expected that its analysis should be easier than for the linearly edge-reinforced random walk, but curiously 
the question regarding recurrence and transience remains completely open for the ORRW. 
Davis noticed in his paper that ORRW has a trivial behavior in dimension one, and conjectured that 
it is recurrent in dimension two. 

It turns out that the ORRW is quite challenging to analyze due to the nature of its interaction and to the lack of monotonicity. Indeed, 
the {\it drift inwards} that we mentioned above means that, when the random walk is on the boundary of its range, 
it is slightly more likely that it goes back inside its range, a fact that could trigger us to think about recurrence. However, the range of the random walk at that place could be of a form such that 
the \emph{drift inwards} translates to a drift \emph{away from $0$}.

Extremely interesting conjectures have been made about the behavior of the ORRW on $\mathbb{Z}^d$, $d\ge3$, which are usually attributed to Vladas Sidoravicius and Vincent Beffara, and independently to Mike Keane. 
They conjectured that on $\mathbb{Z}^d$, $d\ge3$, there exists a phase transition on the strength of the reinforcement parameter $\delta$. 
That is, there should exist a critical parameter $\delta_c$, a priori depending on the dimension, 
such that if $\delta<\delta_c$ then ORRW with parameter $\delta$ is transient, and if $\delta>\delta_c$ then ORRW is recurrent. 

One can then ask finer questions about the model, for instance, regarding the scaling limit of the random walk in the transient regime, or the size and the shape of the range of the random walk in the recurrent regime. 
All these questions are, of course, still very much open.

It seems particularly interesting to try to study the asymptotic shape of the range in the recurrent case. Simulation suggests that there is a certain {\it shape theorem}:
the range $E_n$ of the walk at time $n$, when properly scaled by some polynomial in $n$, seem to converge to a deterministic shape. 
Nothing has been proved in this direction, and we refer the reader to the nice survey by Gady Kozma~\cite{Kozsurvey} where some pictures from simulations are presented. 

\subsection{Other models related to ORRW}
A very nice explanation for why the aforementioned shape theorem result is true was usually given by Vladas by referring to what he called the \emph{Glassy sphere model}. In this model, 
consider spheres of radius $n\ge 1$ simply put inside each other, like Matryoshka dolls. 
Then, start a random walk on $\Z^d$ from the origin which is reflected upon touching the smallest sphere. 
Once the random walk has touched the smallest sphere a number of times that is proportional to its size (i.e.~$n^{d-1}$ for the $n$-th sphere), 
the sphere is destroyed so that the random walk now gets reflected on the next sphere. 
It is straightforward to prove that the random walk in the glassy sphere model is recurrent in any dimensions. 

One could believe that the ORRW for large $\delta$ follows the same behaviour as the glassy sphere model. In fact, if one believes that the range of ORRW for large $\delta$ grows like a ball, 
then once the ORRW has visited all vertices in a ball of radius $n$, it will roughly visit all the edges on the boundary of this ball before going too far away; 
hence it will ``bump'' on the boundary of this ball a number of times that is comparable to the size of the boundary. 
It is not at all clear to us whether this picture really corresponds to the actual behavior of ORRW. Though simulations suggest that this is indeed the case, one cannot 
disregard that simulations may not be very conclusive for model with such strong self interactions.

Other caricature models have been considered in order to try to understand the ORRW. 
Here is another model which Vladas recurrently mentioned and which seems very interesting but very challenging to analyze (we are not sure who this model should be attributed to). 
Consider a semi-infinite cylinder $\left(\mathbb{Z}/N\mathbb{Z}\right)\times \{n:n\le N\}$. 
On every vertex at non-negative height, i.e.~ on $\left(\mathbb{Z}/N\mathbb{Z}\right)\times \{n:0\le n\le N\}$, put a so-called cookie. 
Then, start a random walk coming from $-\infty$. 
This random walk evolves like a simple random walk with the exception that, when it jumps on a vertex $(z,h)$ where there is a cookie, then it 
instantaneously jumps to the vertex $(z,h-1)$ just below it and the cookie disappears. 
It is clear that this random walk is recurrent as it is essentially one dimensional, 
but interesting questions can be asked about the shape created by the remaining cookies. 
Indeed, one can consider the interface between the area without cookies and the area with cookies. 
This interface is intended to provide a simplistic picture of the microscopic behavior of the ORRW close to the boundary of its range for very large $\delta$. 

Note that the interface looks like a function; if we clear the cookie at a given vertex, then all the cookies from vertices below it will be cleared as well by the definition of the dynamics. 
Several questions arise from this model. For instance, stop the random walk when it reaches for the first time the height $N$. Then, how many cookies are left? 
How does the interface look like at that time? What is the height of the lowest remaining cookie? 
It is believed that, when the random walk first reaches height $N$, almost all the cookies have been eaten, with only $o(N)$ cookies remaining. It is also believed that the fluctuations of the interface should be of order $N^{2/3}$. 
A more daring guess would be that the interface, when the random walk first reaches height $N$, is related to KPZ. 

Such questions also inspired Vladas to look at random walk on growing domains. In this case, there is a growing sequence of subsets of $\Z^d$ called \emph{domains} and denoted by $D_0\subset D_1\subset D_2\subset \cdots$,
and a deterministic sequence of times $t_1 < t_2 < \cdots$ 
such that at a time $t\in [t_{i},t_{i+1})$ the random walk jumps according to a simple random walk that is confined to be inside $D_{i}$ (that is, the random walk is reflected at the boundary of $D_i$). So the sequence 
$t_1,t_2,\ldots$ gives the times at which the domain of the random walk grows. This model was studied by Vladas and others in \cite{DHS,DHS2}, and we refer the reader to \cite{Huang} for more recent results.

\subsection{ORRW in other graphs}
We conclude this section by mentioning 
interesting results that have been proved about ORRW in graphs that are not $\Z^d$. 
Indeed, it is interesting to ask whether the phase transition between recurrence and transience can be observed on {\it some} graph.

The ORRW on ladders has been studied, i.e.~on $\mathbb{Z}\times\{1,\dots,k\}$ with $k\ge2$. 
In this case, the ORRW should clearly be recurrent for all values of the parameter $\delta$. 
First, Sellke \cite{Sellke} proved that ORRW is recurrent for $k=2$, and showed that ORRW is recurrent for any $k\ge2$ as long as $\delta$ is small enough. 
Then, Vervoort \cite{Vervoort} wrote a draft paper giving an incomplete proof of recurrence for large reinforcement parameter, which despite having 
some gaps and mistakes, contained a very good core idea. This argument was later on cleaned and completed in \cite{KSS}.

The ORRW has also been analyzed on trees. 
The first result in that direction is the proof of transience on the binary tree in \cite{DKL} for any value of the parameter $\delta$, 
which shows that there is no phase transition on binary trees unlike the conjectured behavior on $\Z^d$. 
The lack of a phase transition has also been established on Galton-Watson trees by Collevecchio \cite{Coll}, who found a very elegant proof through a comparison to a branching process. 
In the hope of observing a phase transition, Kious and Sidoravicius \cite{KS} considered the ORRW on a particular family of trees, which grows only polynomially fast,
and were able to prove the existence of a phase transition on such trees. 
Later, it was proved in \cite{CKS} that the critical parameter $\delta_c$ of the ORRW of any tree is equal to the a quantity that was called the {\it branching-ruin number} of the tree. 
This quantity characterizes the size of the tree at the polynomial scale.

\section{Random memory walk}
Our motivation to study the random memory walk is to compare it to ORRW in high dimensions and with small reinforcement parameter. 
The rough idea is to say that, if the ORRW is transient and if the dimension is large enough, then the loops produced by the range of the ORRW should not be too large, 
and thus the random walk should not get to revisit its range a large number of times. Consequently, the ORRW would behave as if it had a finite random memory (given by the size of the local loops it produces).

We have no intention to argue that the random memory walk has the same behavior as  the ORRW in high dimension; in particular, as we will see in the definition below, the random memory walk 
has a memory that is independent of the range of the walk, which is certainly not the case for the ORRW. 
Nonetheless, one may ask the question of whether the ORRW in high dimensions and for small reinforcement parameter 
shows a similar regenerative structure as the random memory walk studied in the present paper.

Now we define the random memory walk.
As before, we denote the random walk by $(X_n)_{n\ge 0}$ starting from $X_0=0$. 
Let us denote by $R_{n,m}$ the last $m$ edges visited by the random walk at time $n$; that is, 
\begin{align*}
R_{n,m}=\left\{\{x,y\} \colon \ \exists i\in\{n-m+1,\dots, n\} \text{ s.t.~}\{X_{i-1},X_i\}=\{x,y\} \right\},
\end{align*}
with the convention that $R_{n,0}=\emptyset$.
In order to decide its position at time $n+1$, the random memory walk will have access to a memory of random length regarding its past. The length of this memory is given by the random variable
$K_{n}$, where $K_0,K_1,\ldots$ will form an i.i.d.~sequence of nonnegative random variables. 
Then the distribution of the location of the random memory walk at time $n+1$ will depend only on the current location of the walk ($X_{n}$) and on the information (the memory) 
regarding its $K_n$ last positions which is given by $R_{n,K_n}$.

More precisely, define the filtration $\F_n=\sigma((X_i,K_i), i\le n)$, for any $n\ge0$. 
Assuming $X_n=x\in\mathbb{Z}^d$ and $y$ is a neighbor of $x$, i.e. $|x-y|=1$, the next step is distributed according to the following conditional probability:
\begin{equation}
\PR\left[\left.X_{n+1}=y\right|\F_n\right]=\frac{1+\delta\1{\{x,y\}\in R_{n,K_n}}}{\sum_{z:z\sim x} (1+\delta\1{\{x,z\}\in R_{n,K_n}})},
\label{eq:jump}
\end{equation}
where $\delta>0$ is the reinforcement parameter. In other words, the random memory walk defined above jumps like the ORRW but reinforcing only the last edges in the range, 
where the number of edges chosen to be reinforced 
is a random variable that changes at each step and is given by the sequence $(K_n)_n$.

For the moment we will assume that 
\[
\PR[K_0=0]=p_0>0.
\]
The above assumption is not at all essential for the proof and is made here just to simplify the exposition.
Later in Section \ref{extension}, we explain how our proof can be adapted to remove the above assumption. In that section we also discuss a more general version of this model, where the probability of 
jump of the random walk is not given by~\eqref{eq:jump} but is a more general function of $R_{n,K_n}$.

\section{Our results}\label{results}

We are now ready to state our two main theorems. Our first theorem established transience in dimensions at least 3.

\begin{theorem}\label{mainth1}
   Assume that $\mathbb{E}(K_0)<\infty$. Then, the random memory walk $(X_n)_n$ on $\mathbb{Z}^d$, $d\ge3$, is transient almost surely.
\end{theorem}

In our second result, we establish the scaling limit of the random memory walk under stronger assumptions.

\begin{theorem}\label{mainth2}
If $\mathbb{E}(K_0^3)<\infty$, then $(X_n)_n$ satisfies a functional central limit theorem, that is, for any $T>0$,
\[
\left( \frac{X_{\lfloor nt\rfloor}}{\sqrt{n}}\right)_{t\in [0,T]}\Rightarrow \left(B_t\right)_{t\in[0,T]},
\]
where the convergence holds in law, and where $(B_t)_t$ is a non-degenerate $d$-dimensional Brownian motion.
\end{theorem}

It may seem surprising that we requite a finite third moment for the memory in the above result, instead of only a finite second moment. 
However, as we explain later in the paper, it seems that this is the best we can do with the techniques we use.

\section{Regeneration structure induced by the memory and transience of a sub-walk}
The main idea is to focus on the sequence $(K_n)_n$. 
We will define regeneration times, that is, times at which the random walk forgets its past and starts afresh. 
Once we are able to prove that such times happen infinitely often, we will be able to use classical arguments in order to prove Theorems~\ref{mainth1} and~\ref{mainth2}.

Define $\tau_0:=0$ and
\[
\tau_1=\inf\{n>0: K_{n+i}\le i,\ \forall i\ge0\}.
\]
Intuitively, if we consider time as the non-negative reals $\R_+$ and, for each integer $i\geq 0$, we draw a line segment between $i$ and $i-K_i$, 
then $\tau_1$ is the first position such that there is no 
line segment covering the edge $\{\tau_1-1,\tau_1\}$; see Figure~\ref{fig:K}. Note that when the random walk decides to jump from its location at $\tau_1$ to $\tau_1+1$, it does so as a step of simple random walk (that is,
it just chooses a neighbor of $X_{\tau_1}$ uniformly at random and jumps there), and from that time onwards it will not take into account anymore the edges it traversed before time $\tau_1$. Note also that $\tau_1$ necessarily happens at a time 
for which $K_{\tau_1}=0$; that is the reason why we consider the assumption that $\PR[K_0=0]>0$.
%
\begin{figure}[htbp]
   \begin{center}
      \includegraphics[scale=.7]{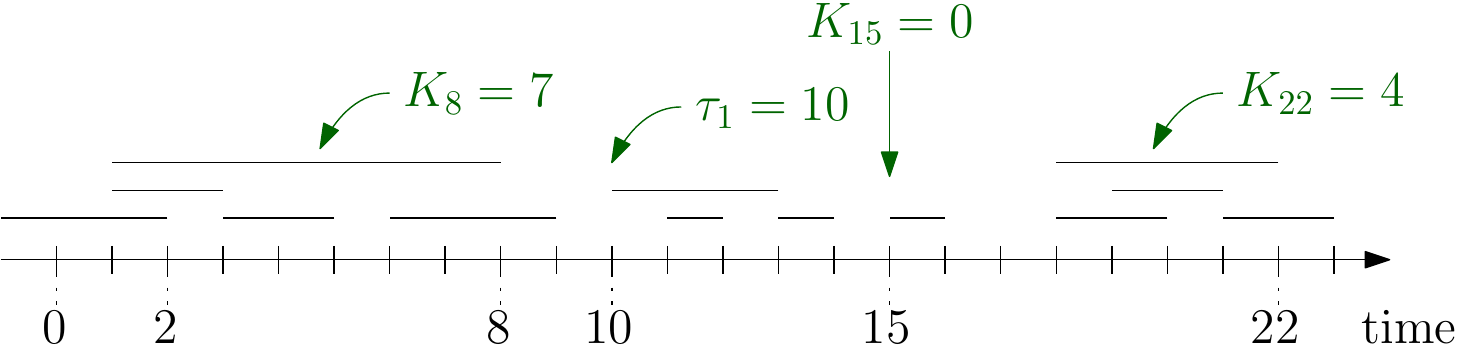}
   \end{center}
   \caption{Illustration of the regeneration time structure. The length of the horizontal line segment ending at coordinate $i$ represents the variable $K_i$ (which shows how far in the past the random walk needs to look 
   at to decide where to be at time $i+1$. Line segments are drawn at different heights for illustration purpose.}
   \label{fig:K}
\end{figure}

Now we show that $\tau_1$ is finite almost surely.
\begin{proposition} \label{moments}
We have that $\PR\left[\tau_1<\infty\right]=1$ if and only if $\mathbb{E}\left[K_0\right]<\infty$.
Moreover, for any integer $n\ge1$, we have that $\mathbb{E}(\tau_1^m)<\infty$ if and only if $\mathbb{E}(K_0^{m+1})<\infty$.
\end{proposition}

\begin{proof}
First, we note that if $\PR\left[\tau_1=1\right]=0$, then $\PR\left[\tau_1<\infty\right]=0$ as, for any $n>1$, we have
\begin{align*}
\PR\left[\tau_1=n\right]&=\PR\left[\forall 1\le k \le n-1, \exists j\ge 0 \text{ s.t.~}K_{k+j}>j, K_{n+i}\le i, \forall i\ge0\right]\\
&\le\PR\left[K_{n+i}\le i, \forall i\ge0\right]\\
&=\PR\left[K_{1+i}\le i, \forall i\ge0\right]\\
&=\PR\left[\tau_1=1\right]=0.
\end{align*}

Now, assume $\PR\left[\tau_1=1\right]>0$ and let us study the event $\{\tau_1<\infty\}$. This event can be seen as successive trials of realizing the events $\{K_{n+i}\le i, \forall i\ge0\}$, and these trials are independent and have probability $\PR\left[\tau_1=1\right]$. 
To prove that properly, let us define recursively
\begin{align*}
   T_1=1, \quad S_k=\inf\{i\ge0: K_{T_k+i}>i\}, \text{ and}\\
    T_{k+1}=T_{k}+S_k+1 \text{, for all }k\ge1.
\end{align*}
In words, $T_k+S_k$ is the first position after $T_k-1$ for which the memory of the walk at that time (equivalently, the line segment that ends there) goes back all the way to $T_k-1$. 
For example, in Figure~\ref{fig:K}, we have that $T_1+S_1=2$, and subsequently we get $T_2=3$ and $T_2+S_2=3$. 

The idea behind this definition is that if $\tau_1>1$, then we look for the value of $S_1$. This translates to checking the random variables $K_{T_1},K_{T_1+1},\ldots$ until finding the value of $S_1$. If we obtain that 
$S_1<\infty$, then position $T_1+S_1+1=T_2$ is a possible candidate for $\tau_1$. If it turns out that $\tau_1 > T_2$, then we look for $S_2$ and $T_3$. 
At each step of this procedure, say step $k\ge1$, 
we will show that, regardless of the values of $T_1,T_2,\ldots,T_{k}$ and regardless of the values of $S_1,S_2,\ldots,S_{k-1}$, with positive probability we have that $S_{k}=\infty$, which in turn gives that 
$\tau_1=T_{k+1}=T_k+S_k+1$.

More formally, define $N=\inf\{k\ge1: S_k=\infty\}$. Using these random variables, we have that $\{\tau_1<\infty\}=\{N<\infty\}$. Also, note that, $\PR[S_1=\infty]=\PR[\tau_1=1]$ and, for any $k>0$, conditional on $T_k$, $S_k$ is distributed like $S_1$.
Hence, one can write
\begin{align*}
\PR[\tau_1<\infty]&=\sum_{k=1}^\infty \PR[N=k]\\
&=\sum_{k=1}^\infty \PR\left[\bigcap_{i=1}^{k-1}\{S_i<\infty\}, S_k=\infty\right]\\
&=\sum_{k=1}^\infty  \PR\left[\tau_1=1\right]\left(1- \PR\left[\tau_1=1\right]\right)^{k-1}
=1.
\end{align*}

Hence, we have proved that if 
$\PR\left[\tau_1=1\right]>0$, then $\PR\left[\tau_1<\infty\right]=1$. 
Finally, we can conclude the first statement of the proposition by noting that
\begin{align*}
\PR\left[\tau_1=1\right]=\prod_{i=0}^{\infty}\PR[K_0\le i]=\prod_{i=0}^{\infty}(1-\PR[K_0> i])\sim ce^{-\mathbb{E}(K_0)},
\end{align*}
and therefore $\PR\left[\tau_1=1\right]>0$ if and only if $\mathbb{E}(K_0)<\infty$.

Now we turn to the second part of the proposition. For this purpose, note that, from the definition of $(T_k)_k$, $(S_k)_k$ and $N$, 
\[
\tau_1=T_N=N+\sum_{k=1}^{N-1}S_k,
\]
where we recall that $N$ is distributed as a geometric random variable with parameter $\PR\left[\tau_1=1\right]$, and where the random variables $S_k$ appearing in the sum above are conditioned to be finite. Hence, $\tau_1$ is essentially equal to a sum of a geometric number of independent random variables distributed like $S_1$ conditioned on $\{S_1<\infty\}$. Therefore, for any $m\ge1$, $\tau_1$ has an $m$-th moment if and only if $S_1$ conditioned on $\{S_1<\infty\}$ has an $m$-th moment. Now, for any $k\ge0$, one can write
\begin{align*}
\PR\left[\left.S_1=k\right|S_1<\infty\right]&=\frac{\PR\left[ K_{1+i}\le i\text{, for }0\le i \le k-1\text{, and } K_{1+k}>k \right]}{\PR[S_1<\infty]}\\
&=\frac{\prod_{i=0}^{k-1}\left(1-\PR\left[ K_{0}> i\right]\right)\times \PR\left[ K_{0}>k \right]}{\PR[S_1<\infty]}.
\end{align*}
Now, assume that $\mathbb{E}(K_0)<\infty$. In that case, as shown above, $\prod_{i=0}^{k-1}\left(1-\PR\left[ K_{0}> i\right]\right)$ converges to a positive constant. 
Besides, we have $\PR[S_1<\infty]=\PR[\tau_1=1]>0$. 
Thus, there exist constants $c_0$ and $c_1$ such that
\begin{align*}
c_0 \PR\left[ K_{0}>k \right]\le \PR\left[\left.S_1=k\right|S_1<\infty\right]\le c_1  \PR\left[ K_{0}>k \right].
\end{align*}
Thus, we have that, for any $m\ge1$,
\begin{align*}
c_0 \sum_{k=1}^\infty k^m \PR\left[ K_{0}>k \right]\le \mathbb{E}\left(\left.S_1^m\right|S_1<\infty\right)\le c_1 \sum_{k=1}^\infty k^m \PR\left[ K_{0}>k \right].
\end{align*}
From there, it is clear that $\mathbb{E}\left(\left.S_1^m\right|S_1<\infty\right)$ is finite if and only if $K_0$ has an $(m+1)$-th moment, which concludes the proof.
\end{proof}

The time $\tau_1>0$ is referred to as  the first \emph{regeneration time}. Let us denote 
\begin{equation}
   \D_n:=\{K_{n+i}\le i,\ \forall i\ge0\}
   \label{eq:dn}
\end{equation}
the event on which $n\ge0$ is a regeneration time.\\
By Proposition \ref{moments}, we have that if $\mathbb{E}(K_0)<\infty$ then $\PR[\D_1]>0$, which easily implies that $\PR[\D_0]>0$.
Therefore,  we can safely define the conditional probability $\Pb[\cdot]:=\PR[\cdot|\D_0]$ and we have that if $\ES(K_0)<\infty$ then
\[
\Pb[\tau_1<\infty]=\frac{\PR[\tau_1<\infty,\D_0]}{\PR[\D_0]}=1.
\]

Also, we have that
\begin{align}\nonumber
\Pb[\D_n]&=\frac{\PR[K_{i}\le i,\ \forall 0\le i\le n-1]\times \PR[K_{n+i}\le i,\ \forall i\ge0]}{\PR[K_{i}\le i,\ \forall i\ge0]}\\ \nonumber
&=\PR[K_{i}\le i,\ \forall 0\le i\le n-1]\\
&\ge \PR[\D_0]>0. \label{ineq}
\end{align}

We inductively define the sequence of regeneration times $\tau_n=\tau_{n-1}+\tau_1\circ\theta_{\tau_{n-1}}$, where $\theta$ is the canonical shift. The following proposition is a classical result on regeneration times.
\begin{proposition}\label{regeneration}
Assume that $\ES(K_0)<\infty$. The random variables  $(X_{\tau_{n}}-X_{\tau_{n-1}},\tau_n-\tau_{n-1})_{n\ge1}$ are independent and, except for $n=1$, are distributed like $(X_{\tau_1},\tau_1)$ under $\Pb$. In particular, all the regeneration times $\tau_n$, $n\ge1$, are finite $\PR$-almost surely.
\end{proposition}

\begin{proof}
This easily follows from general and classical arguments. For instance, one can replicate the proof of Corollary 1.5 in \cite{SznitZern}, which comes from Proposition 1.3 and Theorem 1.4 therein.
\end{proof}

Note that, from the above, we have that if $X$ is almost surely transient under $\Pb$ then it is almost surely transient under $\PR$ as $\liminf||X_n||\ge -||X_{\tau_1}||+\liminf||X_n-X_{\tau_1}||$.
Nevertheless, it is not obvious that $X$ satisfies a $0$-$1$ law for transience, even under $\Pb$.\\

In this section, we want first to prove transience and CLT for the walk $(X_n)$ considered at regeneration times. 
For this purpose, define the walk $(Y_k)_{k\ge0}$ on $\mathbb{Z}^d$ such that $Y_k=X_{\tau_k}$ for any $k\ge0$.
\begin{proposition}
   If $\ES(K_0)<\infty$, then the random walk $(Y_k)$ is transient under $\Pb$, and under $\PR$.
\end{proposition}
\begin{proof}
   Assume $\ES(K_0)<\infty$. From Proposition \ref{regeneration}, we have that, under $\Pb$, 
   $(Y_{k+1}-Y_k)_{k\ge0}$ is a sequence of i.i.d.~random variables.
   As the definition of the walk $(X_n)$ is symmetric with respect to every direction of $\mathbb{Z}^d$, 
   we have that, under $\Pb$, the process $(Y_k)_k$ is a symmetric, genuinely $d$-dimensional random walk. 
   We can then directly conclude the first statement by using Theorem T1, p.83 of Spitzer's book \cite{Spit}.
\end{proof}

\section{Transience and CLT for the Random Memory Walk}

The proof of the CLT (Theorem~\ref{mainth2}) will come easily from classical arguments. 
On the other hand, the proof for transience (Theorem~\ref{mainth1}) 
requires some work as we want to derive it under minimal assumptions. 
The idea is that, once we know that the walk $(Y_k)=(X_{\tau_k})_k$ is transient, 
we need to prove that the random walk $(X_n)$ cannot come back to zero between two regeneration times infinitely often.


\begin{proof}[Proof of Theorem~\ref{mainth1}]
   We will show the transience of $(X_n)$.
   Note that $X$ is transient, i.e.~$||X_n||\to \infty$, if and only if it visits $0$ finitely often. 
   We already know that the random walk $(Y_k)_k$ visits $0$ only finitely often, 
   which is equivalent to saying that there is only a finite number of indices $i$ such that $X_{\tau_i}=0$. 
   We need to prove that $X$ cannot come back to $0$ between two regeneration times infinitely often.
   
   Let us define the sequence of successive return times  to $0$ by $R_0=0$ and $R_i=\inf\{n>R_{i-1}:X_n=0\}$, for $i\ge1$.
   In the following computation, we use the fact that, every time $X$ is back at $0$ and this time is a regeneration time, it implies that $Y$ is back at $0$, thus this cannot happen infinitely often. 
   Recall that $\mathcal{D}_n$ is the event that $n$ is a regeneration time; cf.~\eqref{eq:dn}. We have that
   \begin{align*}
      &\Pb[\cap_{i\ge1}\{R_i<\infty\}]\\
      &=\Pb[\cap_{i\ge1}\cup_{k\ge i}\{R_k<\infty, \mathcal{D}_{R_k}\}]
         +\Pb[\cup_{i\ge1}\cap_{k\ge i}\{R_k<\infty, \mathcal{D}_{R_k}^c\}]\\
      &\le \Pb[Y\text{ visits }0\text{ i.o.}]
      +\Pb[\cup_{i\ge1}\cap_{k\ge i}\{R_k<\infty, \mathcal{D}_{R_k}^c\}]\\
      &=\Pb[\cup_{i\ge1}\cap_{k\ge i}\{R_k<\infty, \mathcal{D}_{R_k}^c\}].
   \end{align*}
   Hence, we obtain the bound
   \begin{align*}
      \Pb[\cap_{i\ge1}\{R_i<\infty\}]\le \sum_{i\ge1}\Pb[\cap_{k\ge i}\{R_k<\infty, \mathcal{D}_{R_k}^c\}].
   \end{align*}
   Let us fix an index $i\ge1$ and prove that $\Pb[\cap_{k\ge i}\{R_k<\infty, \mathcal{D}_{R_k}^c\}]=0$.\\
   We need to define inductively a sequence of stopping times that are all finite on $\cap_{k\ge i}\{R_k<\infty, \mathcal{D}_{R_k}^c\}$. First, define
   \begin{align*}
      \tilde{R}_1&=R_i\ge i\ge1,\\
      \tilde{S}_1&= \tilde{R}_1+\inf\{j\ge 0: K_{\tilde{R}_1+j}>j\}.
   \end{align*}
   Note that $\tilde S_1$ is the first position after $R_1$ whose memory reaches back to before $\tilde R_1$; in other words, $\tilde S_1$ is the first position that shows that $\tilde R_1$ is not a regeneration time.
   Then, define inductively, for any $k\ge1$,
   \begin{align*}
      \tilde{R}_{k+1}&=\inf\{j>\tilde{S}_k: X_j=0\}\ge k+1,\\
      \tilde{S}_{k+1}&= \tilde{R}_{k+1}+\inf\{j\ge 0: K_{\tilde{R}_{k+1}+j}>j\}.
   \end{align*}
   The times $(\tilde{R}_n)_n$ are stopping times with respect to the filtration $\mathcal{F}_n:=\sigma\left(X_k,K_{k-1},0\le k\le n\right)$ 
   and the  times $(\tilde{S}_n)_n$ are stopping times with respect to $\sigma\left(X_k,K_{k},0\le k\le n\right)$. Moreover, we have that
   \begin{align}\nonumber
   \Pb[\cap_{k\ge i}\{R_k<\infty, \mathcal{D}_{R_k}^c\}]&\le \Pb[\cap_{k\ge 1}\{\tilde{R}_k<\infty, \tilde{S}_k<\infty\}]\\
   &=\lim_{N\to\infty} \Pb[\cap_{k= 1}^N\{\tilde{R}_k<\infty, \tilde{S}_k<\infty\}].\label{ineq2}
   \end{align}
   Now, note that, on the event $\{\tilde{R}_k<\infty\}$,
   \begin{align*}
   \Pb\left[\left. \tilde{S}_k<\infty\right|\F_{\tilde{R}_k}\right]&=1-\Pb\left[\left. \D_{\tilde{R}_k}\right|\F_{\tilde{R}_k}\right]\\
   &=1-\sum_{n\ge k}\1{\tilde{R}_k=n}\Pb\left[\left. \D_{n}\right|\F_n\right]\\
   &=1-\sum_{n\ge k}\1{\tilde{R}_k=n}\Pb\left[ \D_{n}\right]\\
   &\le 1-\PR[\D_0],
   \end{align*}
   where we used that $\D_n$ is independent of $\F_n$ and \eqref{ineq}. Together with \eqref{ineq2}, we obtain that
   \begin{align*}
   \Pb[\cap_{k\ge i}\{R_k<\infty, \mathcal{D}_{R_k}^c\}]&\le \lim_{N\to\infty}\left(1-\PR[\D_0]\right)^N=0.
   \end{align*}
   This finally implies that 
   \begin{align*}
   \Pb[X\text{ is recurrent}]=\Pb[\cap_{i\ge1}\{R_i<\infty\}]=0.
   \end{align*}
\end{proof}

\begin{proof}[Proof of Theorem~\ref{mainth2}]
   We now establish the functional central limit theorem, assuming $\ES(K_0^3)<\infty$. 
   We will simply explain why it holds, as this can be proved by following classical results, 
   for instance the proof of Theorem 4.1 in \cite{Sznit} (the only difference is that Brownian motion being non-degenerate comes much more easily in our case, as the process is fully symmetric). 
   The idea of the proof is simply that a functional CLT holds for the random walk $(Y_k)_k$ as, for each $k\ge1$, 
   $Y_k$ is a sum of i.i.d.~random variables which are centered and square integrable (using our assumptions). 
   This comes from Donsker's invariance principle. 
   From there, one only needs an inversion argument for $k\mapsto \tau_k$, 
   which comes from the fact that $\tau_k$ is also a sum of i.i.d.~random variables (satisfying a law of large numbers), 
   and for the first $n$ regeneration times, the distances between successive regeneration times are small compared to $\sqrt{n}$ (in probability).
   This latter step is guaranteed by the fact that $\tau_1$ has a finite second moment under $\Pb$.
\end{proof}

\section{Extensions} \label{extension}

There are two main ways in which our results can be extended. The first one is that the assumption $\PR(K_0=0)>0$ is not necessary. 
The second one is that the jump distribution of the random memory walk does not need to have the form
of a once-reinforced random walk, as stated in~\eqref{eq:jump}.

We start explaining how we can get over the assumption $\PR(K_0=0)>0$. 
This assumption might seem arbitrary at first, but this is actually equivalent to saying that, regardless of the past history of the random walk, 
the walker jumps to any given neighbor with a probability bounded below by a universal constant.
Note that this is indeed the case when the jump distribution is as given by~\eqref{eq:jump} since the probability that the walker jumps to any given neighbor is at least $\frac{1}{1+(2d-1)(1+\delta)}$, regardless of everything else.
So even if we had $\PR(K_0=0)=0$, we could redefine the jump distribution and the distribution of $K_0$ to have $\PR(K_0=0)>0$.

This then leads us to look at different jump distributions for the walker. Consider the following more general version of the random memory walk.
Define the filtration $\F_n=\sigma((X_k,K_k), k\le n)$, for any $n\ge0$. Assuming $X_n=x\in\mathbb{Z}^d$ and $y$ is a neighbor of $x$, i.e. $|x-y|=1$, the next step is distributed according to the following conditional probability:
\begin{equation}\label{def0}
\PR\left[\left.X_{n+1}=y\right|\F_n\right]=f\left(x,y,R_{n,K_n}\right),
\end{equation}
where $f:\mathbb{Z}^d\times \mathbb{Z}^d\times \mathcal{E}\to (0,1)$ is some predetermined function, and $\mathcal{E}$ denotes the set of all finite subsets of edges of $\mathbb{Z}^d$.

Then our proofs work provided $f$ satisfies some symmetry assumption. 
Namely, it is enough to require that $f$ is invariant under graph isomorphism. 
We also want to impose that either $\PR(K_0=0)>0$ or there exists a positive constant $c$ so that for any neighboring vertices $x$ and $y$, and any $R\in\mathcal{E}$ we have 
$$
   f(x,y,R)\ge c.
$$
%

\end{document}